\documentclass[12pt,oneside]{amsart}
\usepackage{graphicx, amssymb,amscd, verbatim}
\usepackage{euscript}

      \theoremstyle{plain}
      \newtheorem{theorem}{Theorem}[section]
      \newtheorem{example}[theorem]{Example}
      \newtheorem{lemma}[theorem]{Lemma}
      
      \newtheorem{proposition}[theorem]{Proposition}
      \newtheorem{remark}[theorem]{Remark}

\numberwithin{equation}{section}

      \makeatletter
      \def\@setcopyright{}
      \def\serieslogo@{}
      \makeatother

\def\natural{{\mathbb N}}
\def\real{{\mathbb R}}
\def\torus{{\mathbb T}}

\def\J{{\mathcal J}}

\def\eg{\emph{e.g., }}

\def\ff{\tilde f}
\def\h{\tilde h}
\def\L{\tilde L}

\def\V{\mathcal V}
\def\Vf{\mathcal V^f}
\def\VL{\mathcal V^L}
\def\Ef{\mathcal E^f}
\def\EL{\mathcal E^L}

\def\Wf{\mathcal W^f}
\def\WL{\mathcal W^L}
\def\Uf{\mathcal U^f}
\def\UL{\mathcal U^L}
\def\F{\mathcal F}

\def\M{\mathcal{M}}
\def\E{\mathcal{E}}
\def\Td{\mathbb T^d}
\def\R{\mathbb R}
\def\Q{\mathbb Q}
\def\C{\mathbb C}
\def\Z{\mathbb Z}
\def\N{\mathbb N}
\def\T{\mathbb T}
\def\dist{\text{dist}}

\def\Id{\text{Id}}

\def\e{\epsilon}

\def\Ci{C^\infty}

\def\QED{\hfill\hfill{\square}}

\begin{document}

\date{\today}
\author{Andrey Gogolev$^\ast$, Boris Kalinin$^{\ast\ast}$, Victoria Sadovskaya$^{\ast\ast\ast}$ \\\\ with an appendix by Rafael de la LLave}

\address{Department of Mathematics,
 University of Texas at Austin, Austin, Texas 78712, USA}
 \email{agogolev@math.utexas.edu, llave@math.utexas.edu}

\address{Department of Mathematics $\&$ Statistics,
University of South  Alabama,  Mobile, AL 36688, USA}
\email{kalinin@jaguar1.usouthal.edu, sadovska@jaguar1.usouthal.edu}

\title[Local rigidity for Anosov automorphisms]
{Local rigidity for Anosov automorphisms}

\thanks{$^\ast$ Supported in part by NSF grant DMS-1001610}
\thanks{$^{\ast\ast}$  Supported in part by NSF grant DMS-0701292}
\thanks{$^{\ast\ast\ast}$ Supported in part by NSF grant DMS-0901842}


\begin{abstract}
We consider an irreducible Anosov automorphism $L$ of
a torus $\T^d$ such that
no three eigenvalues have the same modulus. We show that $L$
is locally rigid, that is, $L$ is $C^{1+\text{H\"older}}$ conjugate
to any $C^1$-small perturbation $f$ such that the derivative
$D_pf^n$ is conjugate to $L^n$ whenever $f^np=p$.
We also prove that toral automorphisms satisfying these
assumptions are generic in $SL(d,\Z)$.
Examples constructed in the Appendix show importance
of the assumption on the eigenvalues.
\end{abstract}

\maketitle


\section{Introduction}

Hyperbolic dynamical systems have been one of the main objects of
study in smooth dynamics. Basic  examples of such systems are given
by {\em Anosov automorphisms} of tori: for a hyperbolic matrix $F$
in $SL(d,\Z)$ the map $F:\R^d\to \R^d$ projects to an automorphism
of the torus $\T^d=\R^d/\Z^d$. More generally, a diffeomorphism $f$
of a compact Riemannian  manifold $\M$ is called {\em Anosov}\, if
there exist a decomposition of the tangent bundle $T\M$ into two
$f$-invariant continuous distributions $E^{s,f}$ and $E^{u,f}$, and
constants $C>0$, $\lambda>0$, such that for all $n\in \N$,
 $$
   \begin{aligned}
  \| Df^n(v) \| \leq C e^{-\lambda n} \| v \|&
      \;\text{ for all }v \in E^{s,f}, \\
  \| Df^{-n}(v) \| \leq C e^{-\lambda n} \| v \|&
      \;\text{ for all }v \in E^{u,f}.
  \end{aligned}
 $$
The distributions $E^{s,f}$ and $E^{u,f}$ are called the stable and unstable distributions~of~$f$.

Structural stability is a fundamental property of hyperbolic systems.
If $g$ is an Anosov diffeomorphism and $f$ is
sufficiently $C^1$ close to $g$, then $f$ is also Anosov and is
topologically conjugate to $g$, i.e. there is a homeomorphism
$h$ of $\M$ such that
$$
g = h^{-1} \circ  f \circ h.
$$
In this paper we study  regularity of the conjugacy $h$. It is well
known that in general $h$ is only H\"older continuous. A necessary
condition for it to be $C^1$
is that  the derivatives of the return
maps of $f$ and $g$ at the corresponding periodic points are
conjugate. Indeed, differentiating $g^n = h^{-1} \circ  f^n \circ h$
at a periodic point $p=f^n(p)$ gives
$$
 D_pg^n = (D_ph)^{-1} \circ D_{h(p)} f^n \circ D_ph.
$$
A diffeomorphism $g$ is said to be {\em locally rigid} if for any
$C^1$-small perturbation $f$ this  condition is  also sufficient for
the conjugacy to be a $C^1$ diffeomorphism. The problem of local
rigidity has been extensively studied and  Anosov systems with
one-dimensional stable and unstable distributions were shown to be
locally rigid \cite{L0,LM, L1,P}.

Local rigidity problem in higher dimensions is much less understood.
Examples where the periodic condition is not sufficient were constructed by
R. de la Llave  \cite{L1,L2}. However, the one-dimensional results
were extended in two directions. In the case when $g$ is conformal
on the full stable and unstable distributions, local rigidity was
established for some classes of systems \cite{L2,KS1,L3,KS3}.

In a different direction, local rigidity was proved in \cite{G} for
an irreducible Anosov toral automorphism $L : \T^d \to \T^d$
with real eigenvalues of distinct moduli, as well as for some
nonlinear systems with similar structure. We recall that $L$ is said
to be {\em irreducible}\, if it has no rational invariant subspaces,
or equivalently if its characteristic polynomial is irreducible over
$\Q$. It follows that all eigenvalues of $L$ are simple.
An important feature of this case is that $\R^d$ splits into a direct sum
of one-dimensional $L$-invariant subspaces. This splitting gives rise
to the corresponding linear foliations on $\T^d$ which are
expanded or contracted by $L$ at different rates.
Such a splitting persist for $C^1$-small perturbations of $L$ and
provides a framework for studying regularity of the conjugacy.
\vskip.1cm

Examples in~\cite{G}  show that irreducibility of $L$ is a necessary
assumption for local rigidity except when $L$ is conformal on
the stable and unstable distributions.
The main result of this paper is the following theorem which
establishes local rigidity for a broad class of  irreducible
toral automorphisms.
We give a concise proof that uses  techniques from \cite{G,L2,KS3}
 along with some new results on conformality
of cocycles from \cite{KS4}.

\begin{theorem} \label{main}
Let $L:\T^d\to\T^d$ be an irreducible Anosov automorphism
such that no three of its eigenvalues have the same modulus.
Let  $f$ be a $C^1$-small perturbation of $L$ such that the
derivative $D_pf^n$ is conjugate to $L^n$ whenever $p=f^np$.
Then $f$ is $C^{1+\text{H\"older}}$ conjugate to $L$.

\end{theorem}

We note that irreducibility of $L$ implies that it is diagonalizable
over $\C$. Hence assuming that $D_pf^n$ is conjugate to $L^n$ is
equivalent to assuming that $D_pf^n$ is also diagonalizable over
$\C$ and has the same eigenvalues as $L^n$. The only extra
assumption in the theorem ensures that the dimensions of the
subspaces in the splitting by rates of expansion/contraction are not
higher than two. It allows $L$ to have pairs of complex conjugate
eigenvalues as well as pairs $\lambda, -\lambda$.
We prove this theorem in Section~2.

In~Section~3, we show that toral automorphisms satisfying the
assumptions of the theorem are generic in the following sense.
Consider the set of matrices in $SL(d,\Z)$ of norm at most $T$.
Then the proportion of matrices corresponding to automorphisms
that do not satisfy our assumptions goes to zero as $T \to \infty$.
Moreover, it can be estimated by $c\,T^{-\delta}$ for some $\delta >0$.

Example~A.3 in the Appendix yields an Anosov toral automorphism
conformal on a three-dimen\-sional invariant subspace and a perturbation
with conjugate periodic data whose derivative is not uniformly quasi-conformal
on the corresponding three-dimensional invariant distribution. This, in
particular, precludes smoothness of the conjugacy. The automorphism
is reducible, so the example does not prove that the extra assumption is
indeed necessary for our theorem. However, it clearly shows that current
methods cannot be pushed further to give the result without this assumption.


\section{Proof of Theorem \ref{main}} \label{proofs}

\subsection{Notation and outline of the proof}
We denote by $E^{s,L}$ and $E^{u,L}$  the stable and unstable
distributions of $L$. Since $f$ is $C^1$ close to $L$, $f$ is also
Anosov and we denote its stable and unstable distributions by
$E^{s,f}$ and $E^{u,f}$. They are tangent to the stable and unstable
foliations $W^{s,f}$ and $W^{u,f}$ respectively (see, e.g.
\cite{KH}). The leaves of these foliations are $C^\infty$ smooth,
but in general the distributions $E^{s,f}$ and $E^{u,f}$  are only
H\"older continuous transversally to the corresponding foliations.
\vskip.2cm

Let $1<\rho_1 <\rho_2< \dots <\rho_l$ be the distinct moduli
of the unstable eigenvalues of $L$, and let
$$
 E^{u,L}= E_1^L \oplus E_2^L \oplus \dots \oplus E_l^L
$$
 be the corresponding splitting of the unstable distribution.

By the assumption,  the distributions $E_k^L$, $k=1, \dots, l$, are
either one- or two-dimensional. As $f$ is $C^1$-close to $L$, the
unstable distribution $E^{u,f}$ splits into a direct sum of $l$
invariant H\"older continuous distributions close to the
corresponding distributions for $L$:
 $$
   E^{u,f}= E_1^f \oplus E_2^f \oplus \dots \oplus E_l^f
$$
(see, e.g. ~\cite[Section 3.3]{Pes}). We also consider the distributions
$$
  E^f_{(i,j)}=E_i^f\oplus E_{i+1}^f\oplus\ldots\oplus E_j^f.
$$
For any $1 < k \le l$, the distribution $E_{(k,l)}^f$
is a fast part of the unstable distribution and thus it integrates
to a H\"older foliation $W_{(k,l)}^f$ with $C^\infty$ smooth leaves
(see, e.g. ~\cite[Section 3.3]{Pes}).
Moreover, the leaves $W_{(k,l)}^f(x)$ depend
$C^\infty$ smoothly on $x$ along the unstable leaves $W^{u,f}$
(see, e.g. ~ \cite[Proposition 3.9]{KS2}).

\vskip.1cm

{\bf Notation.}
 We say that an object is $C^{1+}$ if it is $C^1$ and its differential
  is H\"{o}lder continuous with some positive exponent.
We say that a homeomorphism $h$ is $C^{1+}$ along a foliation
$\mathcal F$ if the restrictions of $h$ to the leaves of $\mathcal
F$ is $C^{1+}$ and the derivative $Dh|_{\mathcal F}$ is H\"older
continuous on the manifold.
\vskip.1cm

For any $1 \le k < l$ the distribution $E^f_{(1,k)}$ is a slow part
of the unstable distribution. It also integrates to an $f$-invariant
foliation $W_{(1,k)}^f$ with $C^{1+}$ smooth leaves. One way to see
this is to view $L$ as a partially hyperbolic automorphism with the
splitting $E^{s,L}\oplus E_{(1,k)}^L\oplus E_{(k+1,l)}^l$. It follows from 
the structural stability of partially hyperbolic systems~\cite[Theorem
7.1]{HPS}  that for a $C^1$-small perturbation $f$ the
``central" foliation survives; that is, $E^f_{(1,k)}$ integrates to
a foliation $W_{(1,k)}^f$. For an alternative simple and short proof
that uses specifics of our setup and also gives unique integrability,
(as opposed to existence of some foliation tangent to $E^f_{(1,k)}$)
see~\cite[Lemma~6.1]{G}.

Thus within the unstable distribution $E^{u,f}$ there are  flags of
weak and strong distributions
$$
E_1^f=E_{(1,1)}^f\subset E_{(1,2)}^f\subset
\ldots \subset E_{(1,l)}^f  =
E^{u,f},
$$
$$
E_l^f=E_{(l,l)}^f\subset E_{(l-1,l)}^f\subset
\ldots \subset E_{(1,l)}^f  =
E^{u,f}.
$$
Since both flags are uniquely integrable and the leaves of the
corresponding foliations are at least $C^{1+}$, for any $1 \le k \le l\,$ the
distribution $E_k^f=E_{(1,k)}^f\cap E_{(k,l)}^f$ also integrates uniquely
to a H\"older foliation
$$V_k^f=W_{(1,k)}^f\cap W_{(k,l)}^f$$
with $C^{1+}$ smooth leaves. Similarly,
the distributions $E_{(i,j)}^f=E_{(1,j)}^f\cap E_{(i,l)}^f$,  $1\le i\le j\le l$,
integrate to H\"older foliations
$$W_{(i,j)}^f=W_{(1,j)}^f\cap
W_{(i,l)}^f.$$

 \vskip.1cm

We use similar notation for the automorphism $L$:
$E_{(i,j)}^L=E_i^L\oplus\ldots\oplus E_j^L, $ and $W_{(i,j)}^L$ and
$V_i^L$ are the linear foliations tangent to $E_{(i,j)}^L$ and
$E_i^L$ respectively. \vskip.2cm

Since $L$ is Anosov and $f$ is $C^1$ close to $L$, there exists a
bi-H\"older continuous homeomorphism $h:\T^d\to \T^d$ close to the
identity in $C^0$ topology such that
  $$
    h \circ L = f \circ h.
  $$
The conjugacy $h$ takes the flag of weak foliations  for $L$ into
the corresponding  weak flag for $f$:

\begin{lemma}
\label{weak_flag_pres}
For any $1\le k\le l,$ $\,h(W_{(1,k)}^L )=W_{(1,k)}^f$.
\end{lemma}

The proof is the same as that of Lemma~6.3 in~\cite{G}. We
give the argument for the reader's convenience.

\begin{proof}
Let $\h$, $\ff$ and $\L$ be the lifts of $h$, $f$ and $L$ to
$\mathbb R^d$. Similarly we use the tilde sign to denote lifts of
various foliations.
\vskip.1cm

Since $\h(\tilde W^{u, L})=\tilde W^{u, f}$ we have that
$\h(\tilde W_{(1,k)}^{L})\subset\tilde W^{u, f}$.
Let $y\in \tilde W^{u, L}(x)$, then
$$
y\in \tilde W_{(1,k)}^L(x) \, \text{ if and only if }\,
\text{dist} (\L^n(x),\L^n(y) )\le(\rho_k+\epsilon)^n\,\text{dist}(x,y) 
\,\text{ for all }n>0,
$$
where $\text{dist}$ is the standard metric on $\mathbb R^d$. Since $\h$ is $C^0$
close to $\Id$ we further obtain that $y\in \tilde W_{(1,k)}^L(x)$
if and only if for all $n>0$
$$
\text{dist}(\ff^n(\tilde h(x)),\ff^n(\tilde h(y)))=
\text{dist}(\h(\L^n(x)),\h(\L^n(y)))\le(\rho_k+\epsilon)^n\,\text{dist}(x,y)+c.
$$
The latter condition in turn is equivalent to $\h(y)\in \tilde
W_{(1,k)}^f(\tilde h(x))$.
\end{proof}
 We note that
Lemma~\ref{weak_flag_pres} holds for any sufficiently $C^1$-small
perturbation of an Anosov automorphism of $\T^d$. \vskip.1cm

The coarse strategy  of the proof of Theorem~\ref{main} is showing
inductively that $h$ is $C^{1+}$ along $W_{(1,k)}^L$ for any $k$ and
thus along $W_{(1,l)}^L=W^u(L)$. By the same argument, $h$ is
$C^{1+}$ along $W^s(L)$ and hence $h$ is $C^{1+}$ by Journ\'e Lemma:

\begin{lemma}[Journ\'e~\cite{J}]  \label{Journe}
Let $\M_j$ be a manifold and $\mathcal F_j^s$, $\mathcal F_j^u$ be
continuous transverse foliations on $\M_j$ with uniformly smooth
leaves, $j=1, 2$. Suppose that $h:\M_1\to \M_2$ is a homeomorphism
that maps $\mathcal F_1^s$  into $\mathcal F_2^s$ and $\mathcal
F_1^u$ into $\mathcal F_2^u$. Moreover, assume that the restrictions
of $h$ to the leaves of these foliations are uniformly $C^{r+\nu}$,
$r\in \mathbb N$, $0<\nu<1$. Then $h$ is $C^{r+\nu}$.
\end{lemma}

The main steps of the proof of the Theorem are the following statements:
\begin{itemize}
\item $h(V^L_i)=V^f_i$;
\item $h$ is a $C^{1+}$ diffeomorphism along $V_i^L$.
\end{itemize}
Their proofs are interdependent and
organized into an inductive process given by Propositions
\ref{smooth along V}  and \ref{fast_to_fast}.

\begin{proposition} \label{smooth along V}
If $\,h(V^L_i)=V^f_i$, then
 $h$ is a $C^{1+}$ diffeomorphism along $V_i^L$.

\end{proposition}

The proof of this proposition is given  in Section~\ref{smooth
along V proof} below. Since $V_1^L=W_{(1,1)}^L$,
Lemma~\ref{weak_flag_pres} implies that $h(V_1^L)=V_1^f$, and then
Proposition~\ref{smooth along V} yields  that $h$ is $C^{1+}$ along
$V_1^L$. This provides the base of the induction. The inductive step
is given by the following proposition.

\begin{proposition}\label{fast_to_fast}
Suppose that  $h(V_i^L)=V_i^f$, $1\le i\le k-1$, and $h$ is a
$C^{1+}$ diffeomorphism along $W^L_{(1,k-1)}$. Then $h(V_k^L)=V_k^f$
and  $h$ is a $C^{1+}$ diffeomorphism along $W_{(1,k)}^L$.
\end{proposition}

The proof of this proposition is given in
Section~\ref{fast_to_fast proof} (and also uses an inductive
argument). In the proof we only need to establish that
$h(V_k^L)=V_k^f$. Then Proposition~\ref{smooth along V} implies
the smoothness of $h$ along $V_k^L$, and the smoothness along
$W_{(1,k)}^L$ follows from the Journ\'e Lemma \ref{Journe}.

\vskip.4cm

\subsection{Proof of Proposition~\ref{smooth along V}}
\label{smooth along V proof}
In this subsection we  write
$$
  \VL \overset{\text{def}}{=}V_i^L,  \quad
  \Vf\overset{\text{def}}{=}V_i^f, \quad
  \EL\overset{\text{def}}{=}E_i^L=T\VL, \quad
  \Ef\overset{\text{def}}{=}E_i^f=T\Vf.
$$
The proof is an adaptation of arguments of de la Llave~\cite{L2}.
First we show that $h$ is Lipschitz along $\VL$ as a limit of smooth
maps with uniformly bounded derivatives. Then we prove that the
measurable derivative of $h$ along $\VL$ is actually H\"older
continuous. Both steps use Liv\v{s}ic Theorem for commutative and
noncommutative cocycles and rely on conformality of $L$ and $f$
along $\VL$ and $\Vf$ respectively. Conformality of  $f$ along $\Vf$
is crucial and to establish it we use a result from \cite{KS4}.

First we construct a map $h_0$ close to $h$ and
satisfying  the following conditions:
\begin{enumerate}
\item $h_0(\VL)=\Vf$, moreover, $h_0(\VL(x))=\Vf(h(x))$ for all $x$ in $\T^d$;
\item $\sup_{x\in\mathbb T^d} d_{\Vf}(h_0(x),h(x))<+\infty$,
where $d_{\Vf}$  is the distance along the leaves;
\item $h_0$ is $C^{1+}$ diffeomorphism along the leaves of $\VL$.
\end{enumerate}
\vskip.1cm

\noindent Let $\bar{\V}^L$ be the linear integral
foliation of
$E^{s,L}\oplus E_1^L\oplus\ldots\oplus E_{i-1}^L\oplus
E_{i+1}^L\oplus\ldots\oplus E_l^L$.
We define the map
$h_0$ by intersecting  local leaves:
$$
h_0(x)=\V^{f,loc}(h(x))\cap\bar{\V}^{L, loc}(x).
$$
The map is well-defined and satisfies (2) since $h$ is close
to the identity.
Condition (1) holds since $h(\VL(x))=\Vf(h(x))$ by the assumption,
and (3) is satisfied since for any $x$ the leaf $\Vf (h(x))$ is
$C^{1+}$ and $C^1$ close to $\VL(x)$.

It follows easily as in \cite[Theorem~2.1]{L2} that
$$
h=\lim_{n\to\infty}h_n, \; \text{ where }\;
h_n=f^{-n}\circ h_0\circ L^n.
$$
Indeed, let us endow the space of maps satisfying (1) and (2) with
the metric  $d(k_1, k_2)=\sup_x d_{\Vf}(k_1(x), k_2(x))$. Then,
since $f^{-1}$ contracts the leaves of $\Vf$,  it follows that the
map $k \mapsto f^{-1}\circ k \circ L$ is a contraction with the
fixed point $h$.

\vskip.2cm

Now we prove that $h$ is Lipschitz along $\Vf$.
For this it suffices to show that the derivatives
of the maps $h_n$ along $\VL$ are uniformly bounded.
We estimate
$$
\begin{aligned}
 \|D_{\VL(x)}h_n\|
 & \le \| Df^{-n}|_{\Ef(h_0(L^nx))}\|\cdot \|D_{\VL(L^nx)}h_0\| \cdot \|L^n|_{\EL(x)} \| \\
 & = \| \left( Df^n|_{\Ef(f^{-n}(h_0(L^nx)))} \right)^{-1}\|\cdot \|D_{\VL(L^nx)}h_0\| \cdot \|L^n|_{\EL} \| \\
 & \le \| \left( Df^n|_{\Ef(h_n(x))} \right)^{-1}\| \cdot \|L^n|_{\EL} \|
 \cdot  \sup_z \|D_{\VL(z)}h_0\|. \\
 \end{aligned}
 $$
Since  $D_{\VL}h_0$ is continuous on $\T^d$, the supremum
on the right is finite. Now we show that the product
$\| (Df^{n}|_{\Ef(y)})^{-1}\|\cdot \|L^n|_{\EL} \|$ is uniformly bounded
in $y$ and $n$.

\vskip.1cm

We concentrate on the case when $\Vf$ is two-dimensional. The
one-dimensional case is similar except for conformality of $L$ along
$\VL$ and of $f$ along $\Vf$ is trivial. Since $L$ is irreducible it
is diagonalizable over $\C$. Therefore, as the eigenvalues of
$L|_{\EL}$ have the same modulus, $L|_{\EL}$ is conformal with
respect to some norm on $\EL$. We can assume that our background
norm $\|\cdot\|$ is chosen so that $L|_{\EL}$ is conformal.

By the assumption of the theorem, $D_pf^n$ is conjugate to $L^n$
whenever $f^np=p$. It follows that $D_pf^n|_{\Ef(p)}$ is also
diagonalizable over $\C$ and has  eigenvalues of the same modulus.
To obtain conformality of $Df|_{\Ef}$, we apply  the following result
to vector bundle $\E=\Ef$ and cocycle $F=Df|_{\Ef}$.

\vskip.2cm

\noindent \cite[Theorem 1.3]{KS4}
{\it  Let $\E$ be a H\"older continuous linear bundle with
two-dimensional fibers over a
compact Riemannian manifold $\M$.
Let $F: \E \to \E$ be a H\"older continuous linear cocycle over
a transitive Anosov diffeomorphism $f:\M\to\M$.
 If for each periodic point $p\in \M$, the return
map $F^n_p:\E_p\to \E_p$ is diagonalizable over $\C$ and its
eigenvalues are equal in modulus, then $F$ is conformal with respect
to a H\"older continuous Riemannian metric on $\E$.
}
\vskip.2cm

We denote by $\| \cdot \|_x^f$ the norm induced by the metric on
$\Ef(x)$ given by the theorem. The conformality of $Df|_{\Ef}$ with
respect to this norm means that
$$
    \| Df(v) \|_ {f(x)}^f = c(x) \cdot \|v\|_x^f \quad
    \text{for any } x\in \T^d \text{ and }v\in \Ef(x).
$$
Clearly,  $c(x)=\|Df|_{\Ef(x)} \|^f$, the norm of $Df:(\Ef(x),
\|\cdot\|_x^f)\to (\Ef(f(x)), \|\cdot\|_{f(x)}^f)$.

We  set
$$
   a(x) =  \|L|_{\EL(x)} \| =  \|L|_{\EL}\| \quad\text{and} \quad
   b(x) =  c(h(x))=\|Df|_{\Ef(h(x))} \|^f.
$$
The function $a(x)$ is constant in our context, however we will
keep the variable for consistency with $b(x)$.
Since $L$ is conformal on $\EL$, $\,a(x)$  satisfies
$$
a_n(x)\overset{\text{def}}{=}  a(x)a(Lx) \cdots  a(L^{n-1}x) =
 \|L^n|_{\EL} \|.
$$
The function $b(x)$ is H\"older continuous, and using the relation
$f^m\circ h =h\circ L^m$
and the conformality of $Df|_{\Ef}$ we obtain
$$
\begin{aligned}
& b_n(x)\overset{\text{def}}{=} b(x)b(Lx) \cdots  b(L^{n-1}x) = \\
& = \|Df|_{\Ef(h(x))} \|^f \cdot \|Df|_{\Ef(h(Lx))} \|^f\cdots
\|Df|_{\Ef(h(L^{n-1}x))} \|^f  =  \|Df^n|_{\Ef(h(x))} \|^f.
\end{aligned}
$$
We claim that the functions $a$ and $b$ are cohomologous,
i.e. the exists a continuous function $\phi:\T^d\to\R_+$ such that
 $$
 a(x)/b(x)=\phi(Lx)/\phi(x).
 $$
 This follows from the Liv\v{s}ic Theorem~\cite{Liv}, \cite[Theorem 19.2.1]{KH} once we show
 that $a_n(p)=b_n(p)$ for any
periodic point $p=L^np$. We note that $b_n(p)= \|Df^{n}|_{\Ef(h(p))}
\|^f$ is the modulus of the eigenvalues of
$Df^{n}|_{\Ef(h(p))}$ since this linear map is conformal with
respect to norm $\| \cdot \|^f_{h(p)}$. A similar statement holds
for $a_n(p)$ and $L^{n}|_{\EL}$. The coincidence of the periodic
data for $f$ and $L$ implies that indeed $a_n(p) = b_n(p)$ and hence
the functions $a$ and $b$ are cohomologous.  Using conformality  we
obtain that
$$
\begin{aligned}
\|L^n|_{\EL} \| \cdot \| (Df^{n}|_{\Ef(h(x))})^{-1}\|^f &=
\|L^n|_{\EL} \| \cdot \left( \| Df^{n}|_{\Ef(h(x))}\|^f\right)^{-1} =\\
&= a_n(x) / b_n(x) =  \phi(L^nx) /\phi(x)
\end{aligned}
$$
is uniformly bounded since $\phi$ is continuous on $\T^d$. Since the
norm $\| \cdot \|^f$ is  equivalent to $\|\cdot\|$ we obtain that
$\| (Df^{n}|_{\Ef(y)})^{-1}\|\cdot \|L^n|_{\EL} \|$  is uniformly
 bounded in $y$ and $n$. We conclude that $\|D_{\VL(x)}h_n\| $
 is uniformly bounded in $x$ and $n$, and hence $h$ is
 Lipschitz along $\Vf$.

A similar argument shows that
$\|(D_{\VL(x)}h_n)^{-1}\|$ is uniformly bounded and hence
$h$ is bi-Lipschitz along $\Vf$. In particular, $D_{\VL}h$ exists and is
invertible almost everywhere.

\vskip.2cm

Differentiating $f\circ h= h \circ L$ along $\VL$ on
a set of full Lebesgue measure we obtain
$$
 Df|_{\Ef(h(x))} \circ D_{\VL(x)}h = D_{\VL(Lx)} h \circ L |_{\EL(x)} ,
$$
i.e., the cocycles $Df|_{\Ef(h(x))}$ and $L |_{\EL(x)}$ are
cohomologous with transfer function $D_{\VL(x)}h$. The bundle $\Ef$
is trivial since it is close to the trivial bundle $\EL$. Therefore,
$Df|_{\Ef(h(x))}$ and $L |_{\EL(x)}$ can be viewed as  H\"older
continuous $GL(2,\R)$-valued cocycles over the automorphism $L$.
Moreover, the existence of conformal metrics implies that they are
cohomologous to cocycles with values in the conformal subgroup. We
remark that in general measurable transfer functions are not
necessarily continuous \cite[Section 9]{PW}. However, for conformal
cocycles the measurable transfer function coincides almost
everywhere with a H\"older continuous one. This follows from
\cite[Theorem 6.1]{Sch} or from \cite[Theorem 1]{PP} after reducing
cocycles to orthogonal ones by factoring out the norms. See also
\cite{S} for stronger results on $GL(2,\R)$-valued cocycles. We
conclude that $D_{\VL(x)}h$ is H\"older continuous, and hence  $h$
is a $C^{1+}$ diffeomorphism along $\VL$. $\QED$


\subsection{Proof of Proposition~\ref{fast_to_fast}}
\label{fast_to_fast proof}
The proof  is based on the following proposition.

\begin{proposition} \label{prop_for_ind}
Assume that $h(W_{(i,k)}^L)=W_{(i,k)}^f$, $h(V_i^L)=V_i^f$ and $h$
is a $C^{1+}$ diffeomorphism along $V_i^L$. Then $h(W_{(i+1,k)}^L)=W_{(i+1,k)}^L$.
\end{proposition}

We apply Proposition \ref{prop_for_ind}  inductively
with $i=1,\dots, k-1$.
At every step the assumption of the proposition is
fulfilled due to the assumptions in the Proposition~\ref{fast_to_fast}
and the conclusion of Proposition  \ref{prop_for_ind}
at the previous step. We obtain the conclusion of
Proposition~\ref{fast_to_fast}
at the final step when $W_{(i+1,k)}^L=W_{(k,k)}^L=V_k^L$.
\vskip.1cm

It remains to prove Proposition  \ref{prop_for_ind}. We will use the
following simplified notation:
\begin{multline*}
\qquad (\WL,\VL,\UL)=(W_{(i,k)}^L, V_i^L, W_{(i+1,k)}^L), \\
\qquad\shoveleft{(\Wf,\Vf,\Uf)=(W_{(i,k)}^f, V_i^f, W_{(i+1,k)}^f).
\hfill}
\end{multline*}
We note that $\VL$ and $\UL$ are slow and fast sub-foliations
of $\WL$ respectively. Similarly, $\Vf$ and $\Uf$ are slow and fast sub-foliations
in $\Wf$. We also note that $\,\Uf=\Wf \cap  W_{(i+1,l)}^f$. The
foliation $W_{(i+1,l)}^f$ is a fast part of the unstable foliation and
hence is $\Ci$ inside the unstable leaves, see for example
\cite[Proposition~3.9]{KS2}. Therefore, the foliation $\,\Uf$ is $C^{1+}$
inside the leaves of $\Wf$ and the holonomies between
the leaves of $\Vf$ along $\Uf$  are uniformly $C^{1+}$.

Let $\F=h^{-1}(\Uf)$. Then $\F$ is a continuous foliation with
continuous leaves that subfoliates $\WL$.
We need to show that $\F=\UL$. Since $\VL=h^{-1}(\Vf)$, $\F$ is
topologically transverse to $\VL$, that is, any  leaf of $\F$
and and any leaf of $\VL$ in the same leaf of $\WL$ intersect at
exactly one point.
 First we prove an auxiliary statement
that gives some insight into relative structure of $\F$ and $\VL$.

For any point $a\in\Td$ and any $b\in\F(a)$ we denote by $H_{a,b} :
\VL(a) \to \VL(b)$ the holonomy along the  foliation $\F$, i.e. for
every $x\in\VL(a)$ we define $H_{a,b}(x)$ to be the unique point of
intersection $\F(x)\cap\VL(b)$.

\begin{lemma}\label{translation}
For any point $a\in\Td$ and any $b\in\F(a)$ the holonomy map
$H_{a,b}$  is a restriction to $\VL(a)$ of a parallel translation inside $\WL$.
\end{lemma}

\begin{proof}

For any point $c\in\Td$ and any $d\in\Uf(c)$ we denote by
$\tilde H_{c,d} : \Vf(c) \to \Vf(d)$ the holonomy along the foliation $\Uf$,
which is $C^{1+}$ as we noted above.
Since $\F=h^{-1}(\Uf)$ and $h(\VL)=\Vf$ we have
$$
H_{a,b} = h^{-1} \circ \tilde H_{h(a),h(b)} \circ h.
$$
 Since $h$ is a $C^{1+}$ diffeomorphism along $\VL$
we conclude that $H_{a,b}$ is $C^{1+}$.

\begin{figure}[htbp]
\begin{center}
\includegraphics{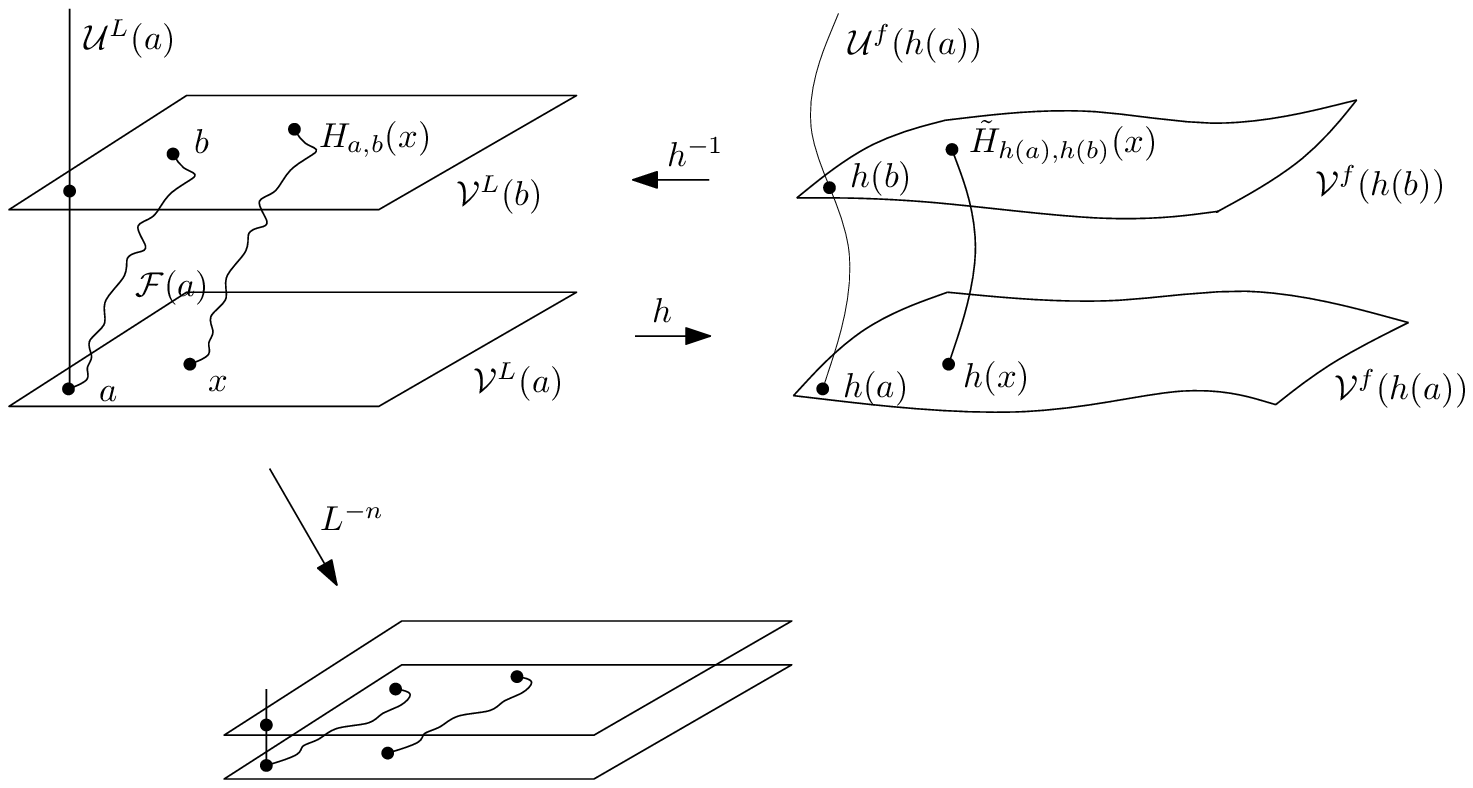}
\end{center}
\label{pic1} \caption{}
\end{figure}

To show that $H_{a,b}$ is the restriction of a parallel translation,
we prove that the differential $DH_{a,b}=\Id$.
We apply  $L^{-n}$ and denote $a_n=L^{-n}(a)$,
$b_n=L^{-n}(b)$. Since $\F=h^{-1}(\Uf)$  and
$f$ preserves the foliation $\Uf$,
$L$ preserves  $\F$ and we can write
$$
H_{a,b} = L^n\circ  H_{a_n,\,b_n} \circ L^{-n}.
$$
Differentiating and denoting  $D_{a_n}H_{a_n,b_n}=\Id+\Delta_n$
we obtain
$$
D_a H_{a,b} = L^n|_{\VL} \circ  D_{a_n}H_{a_n,b_n} \circ L^{-n}|_{\VL}
=  \Id + L^n|_{\VL}\circ \Delta_n \circ L^{-n}|_{\VL}.
$$
Since $L$ is conformal on $\VL$ with respect to some inner product,
$$
   \|L^n|_{\VL}\| \cdot\| L^{-n}|_{\VL}\| \le C  \; \text{ and hence }\;
   \|L^n|_{\VL}\circ \Delta_n \circ L^{-n}|_{\VL}\|\le C\|\Delta_n\|
   \;\text{ for all }n.
$$
It remains to show that $\|\Delta_n\| \to 0 $. This follows easily by
differentiating the equation
$$
  H_{a_n,b_n} = h^{-1} \circ \tilde H_{h(a_n),\,h(b_n)} \circ h.
$$
Indeed, we obtain
$$
  D_{a_n}H_{a_n,b_n} = (D_{b_n}h)^{-1} \circ
  D_{h(a_n)}\tilde H_{h(a_n),\,h(b_n)} \circ D_{a_n} h.
$$
Since $\dist(a_n, b_n)\to 0$ as $n\to \infty$ we obtain that
$D_{h(a_n)}\tilde H_{h(a_n),\,h(b_n)} \to \Id$. 
Also, by choosing a subsequence if necessary, we can assume that
$\lim a_n=\lim b_n =c \in \T^d$. Thus 
$D_{a_n}h \to D_{c}h$, $ D_{b_n}h \to D_{c}h$ and hence $\Delta_n=D_{a_n}H_{a_n,b_n} - \Id \to0$ as $n\to \infty$.
\end{proof}

Now let $a$ be a fixed point of $L$ and  let $B$ be the unit ball
in $\UL(a)$ centered at~$a$. If $B\subset\F(a)$, then $\UL(a)=\F(a)$
by invariance under $L$. Since $L$ is irreducible and
$\UL$ is invariant, the leaf $\UL(a)$  is dense in $\T^d$.
It follows that the set of points $x$ such that
$\UL(x)=\F(x)$ is dense in $\T^d$ and hence $\UL=\F$.
Therefore, it suffices to show that $B\subset\F(a)$.

We argue by contradiction. Assume that there is $z_1\in B$ such that
$z_1\notin\F(a)$. Let
$$
x_1=\VL(z_1)\cap\F(a).
$$
Since $\VL$ has dense leaves we can choose a sequence
$\{b_n,n\ge1\}\subset\VL(a)$ so that $b_n\to x_1$ as $n\to\infty$. Let
$$
y_n=H_{a,x_1}(b_n),
$$
where $H_{a,x_1}$ is the holonomy map along $\F$ from $\VL(a)$ to $\VL(x_1)$.
Continuity of $\F$ implies that the sequence $y_n$ converges to a point
$x_2  \in \F(a)$. Moreover, Lemma~\ref{translation} implies that $\{x_1,x_2\}$
is a parallel  translation of $\{a,x_1\}$.

\begin{figure}[htbp]
\begin{center}

\includegraphics{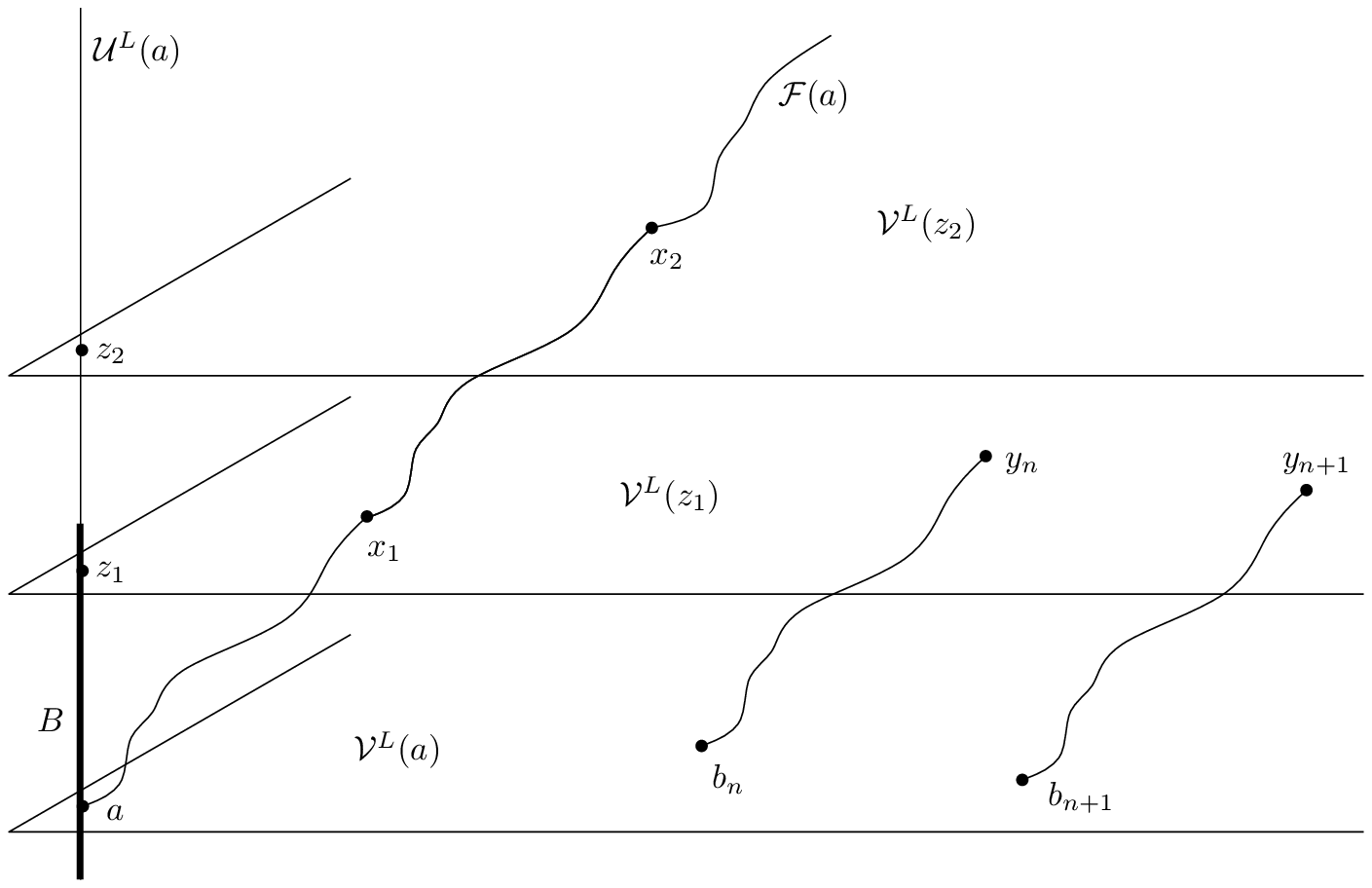}
\end{center}
\label{pic2} \caption{}
\end{figure}

We continue this procedure inductively to construct the sequence
$\{x_n,n\ge 1\}\subset\F(a)$. Let
$$
z_n=\VL(x_n)\cap\UL(a).
$$
Then according to the construction
\begin{equation}\label{ladder}
 d_{\UL}(z_n,a)=n\cdot d_{\UL}(z_1,a) \quad\text{and}\quad
d_{\VL}(x_n,z_n)=n\cdot d_{\VL}(x_1,z_1).
\end{equation}

For every $n$ we take $N(n)$ to be the smallest integer such that
$L^{-N(n)}(z_n)\in B$. Since $L^{-1}$  contracts $\UL$ exponentially
faster than $\VL$, equation \eqref{ladder}
implies that
$$
d_{\VL}(L^{-N(n)}(x_n), L^{-N(n)}(z_n))\to\infty \quad \text{ as } \; n\to\infty.
$$
This contradicts an obvious bound due to compactness of $B$:
$$
\max_{z\in B}d_{\VL}(z,\VL(z)\cap\F(a))<\infty.
$$
Thus we conclude that $\F=\UL$.
$\QED$


\section{Genericity}

In this section we show that  toral automorphisms satisfying
the assumptions of Theorem \ref{main} are generic in $SL(d,\Z)$.
We would like to thank A. Gorodnik,  P.~Sarnak,
and D. Speyer for helpful discussions on this topic.
\vskip.1cm

We consider $SL(d,\R)$, $d\ge 2$, as a subset of the Euclidean
space  of $d\times d$ matrices  equipped with the norm
$\|A\|=(\text{Tr}(A^*A))^{1/2}$. We denote
$$
B_\R (T)=\{A\in SL(d,\R):\|A\|\le T\} \quad \text{and} \quad
B_\Z(T)=\{A\in SL(d,\Z):\|A\|\le T\}.
$$
It is known~\cite[Theorem~3.1]{DRS} that number of
matrices  in $B_\Z(T)$  grows as the Haar volume of $B_\R(T)$.
More precisely
$$
\# B_\Z(T) \sim vol(B_\R(T)) \sim c\,T^{d^2-d}.
$$


Let $E(T)$ be the subset of $B_\Z(T)$ that consists of  matrices
that do {\em not}\, satisfy the assumptions  of Theorem~\ref{main},
i.e. are either reducible (over $\Q$), or non-Anosov, or have at least three
eigenvalues of the same modulus.

\begin{proposition}\label{asymp}
There exists $\delta>0$ such that
$\; \# E(T) \ll T^{d^2-d-\delta}.$
\end{proposition}

To prove the above proposition we first show the following.

\begin{lemma} \label{generic}
The set $E=\cup_{T>0}E(T)$ lies in a finite union of
algebraic hypersurfaces in $SL(d,\R)$.
\end{lemma}

\begin{proof}
We consider $A \in SL(d,\Z)$ and denote its eigenvalues
 by $r_1 , r_2 , ..., r_d$. We will describe explicit relations
 on the entries of matrices in $E$ as symmetric polynomials in
 $r_1 , r_2 , ..., r_d$.
Since the eigenvalues are the roots of the characteristic polynomial $\chi_A$,
such a polynomial can be expressed as a polynomial in the
coefficients of $\chi_A$, and hence as one in the entries of $A$.

Suppose that $A$  is irreducible and  $\chi_A$ has three roots 
of the same modulus,
in particular, $d\ge 3$.  Then $\chi_A$ must have
either two pairs of complex conjugate eigenvalues of the same modulus
or a pair of complex conjugate eigenvalues of the same modulus
as a real eigenvalue. In the first case, the eigenvalues satisfy
$$
P=\prod _{i,j,k,l} (r_i r_j-r_k r_l)=0, \quad
\text{where } 1 \le i,j,k,l \le d  \,\text{ are  distinct},
$$
and in the second case they satisfy
$$
P=\prod _{i,j,k} \,(r_i r_j-r_k^2)=0, \quad
\text{where } 1 \le i,j,k \le d  \,\text{ are  distinct.}
$$

Suppose that $A$ is not Anosov, i.e. it has an eigenvalue of modulus 1.
If $A$ has a complex pair $r_i,r_j$ on the unit circle, then
$r_i  r_j=1$ and the same holds for the product of all other
eigenvalues since $\det A =1$. Thus we obtain a symmetric polynomial
relation
$$
P=\prod_{i\ne j} ( r_i r_j-\prod_{k\ne i,j} r_k ) =0.
$$
Similarly, if  $r_i=1$ or $r_i=-1$ for some $i$,  we have
$$
P=\prod_{i} ( r_i -\prod_{k\ne i} r_k ) =0.
$$
These relations are non-trivial if $d\ge3$. For $d=2$, $A$ is not
Anosov if  and only if $| Tr A| \le 2$. Such matrices lie in
affine hyperplanes $Tr A = k$, $k=0,\pm1,\pm2$.

Finally, suppose that $A$ is reducible, i.e. its characteristic
polynomial $\chi_A$ is reducible over $\Q$.
Since $A$ is in $SL(d,\Z)$, $\chi_A$ is reducible over $\Z$ and
the factors have constant terms equal to $1$ or $-1$. It easy to see
that having such a factor of a given degree imposes a nontrivial
constrain on the coefficients of $\chi_A$.  Alternatively, one can
give relations on the roots as before. For example, if $\chi_A$ has
a factor of degree $k$ with constant term  $1$, then
some product of $k$ eigenvalues is $1$ and hence
the eigenvalues satisfy the relation
$$
P=\prod _{i_1,..., i_k} \,(r_{i_1} \cdots r_{i_k} -1)=0, \quad
\text{where } 1 \le i_1,..., i_k \le d  \,\text{ are  distinct.}
$$
\end{proof}

Now we deduce Proposition~3.1 from the following result
by A. Nevo and P.~Sarnak, which is a particular case of
Lemma~4.2 in \cite{NS}.

\begin{lemma}\label{NS} Let $G$ be a subgroup of $GL(m,\R)$
isomorphic to $SL(d,\R)$. Let $v\in \Z^m$ and $V=Gv$
be the orbit of $v$. Assume that there is a polynomial
$P\in\Q[x_1,..., x_m]$ that does not vanish identically on $V$.
Then there exists $\delta=\delta(P)>0$ such that
$$
\#\{A\in G \cap GL(m,\Z):\, \|A\|\le T, \;P(Av)=0\}\ll
T^{d^2-d-\delta}.
$$
\end{lemma}

We apply this proposition with $m=d^2$ and identify $\R^m$ with
$\text{Mat}_{d\times d}(\R)$ as follows,   first $d$ coordinates are
identified with the first column, next $d$ coordinates with the second
column, etc. We embed $SL(d,\R)$ into $GL(m,\R)$ diagonally
$g\mapsto g\times g\times\ldots\times g$ ($d$ times). It is
easy to see that under these identifications, the action of $SL(d,\R)$ 
on $\R^m$ is the same as the matrix multiplication on the left in 
$\text{Mat}_{d\times d}(\R)$.

We take $v\in \Z^m =\text{Mat}_{d\times d}(\Z)$ to be the identity
matrix. Then $V$ is identified with $SL(d,\R)$ and applying 
Lemma~\ref{NS} to the polynomials given in
Lemma~\ref{generic} yields Proposition~\ref{asymp}. We note that the
norm in Lemma~\ref{NS} comes from $GL(m,\R)$ and is different from
the norm we have defined on $SL(d,\R)$. However this does not make a
difference for the asymptotics since the norms are equivalent.


\appendix
\section{Some examples}
\begin{center}{\bfseries by Rafael de la Llave}
\end{center}

We consider matrix cocycles over an Anosov diffeomorphism $g$
of a manifold $\M$. Such a cocycle is given by a function
$A : \M \to GL(d,\R)$.
Our goal is to construct examples of (a) cocycles which are
conformal at periodic points but are not uniformly quasi-conformal
and (b) Anosov diffeomorphisms such that the restriction of the
derivative to an invariant distribution gives a cocycle as in (a).

For a matrix $A \in GL(d,\R)$ we denote by $K(A)= \|A\| \cdot \|A^{-1}\|$
its {\it quasi-conformal distortion}\, with respect to a norm $\| . \|$ on $\R^d$.
$A$ is called {\it conformal}\, with respect to a given norm if $K(A)=1$.
For example, if $A$ is diagonalizable over $\C$ and all its eigenvalues
are of the same modulus, then $A$ is conformal with respect to a
norm given by a diagonalization of $A$.  We say that a cocycle
$A: \M \to GL(d,\R)$ is {\em uniformly quasi-conformal}\,
if $K(A^n(x))$ is uniformly bounded in $x$ and $n$, where
$$
A^n(x) = A(g^{n-1} x)\cdots A(gx) A(x).
$$
Unlike conformality, uniform quasi-conformality does not depend
on the choice of a norm.

Examples of (a) were already constructed in~\cite{KS4} but
did not give rise to examples of (b).
We also note that our examples are contained in infinite
dimensional families which include linear automorphisms of the tori,
and so they show that parametric rigidity is also impossible.

\begin{example}\label{cocycle}
Let $g$ be an Anosov diffeomorphism  of a manifold $\M$. There
exists a family of $SL(3, \real)$-valued cocycles
$A_\epsilon$, $|\epsilon | < 1$, over $g$ such that:
\begin{itemize}
\item
$A_0$ is a  constant conformal matrix;
\item
$A_\epsilon(x)$ is jointly analytic in $\epsilon$ and $x$;

\item For any $\epsilon$ and any periodic point $p=g^np$,
$A_\epsilon^n (p)$ is conformal in some norm;
\item For any $\epsilon \ne 0$, the cocycle $A_\epsilon$ is not uniformly quasi-conformal.
\end{itemize}
\end{example}

Note that in Example~\ref{cocycle} we can take $g$ to be any Anosov
diffeomorphism and we do not even require transitivity.

\begin{remark}
\label{dimension} One can construct similar examples
taking values in $SL(d, \real)$ for $d \ge 3$. It was shown
in \cite[Theorem 1.3]{KS4} that there are no such examples when
$d =2$, and the result is trivial when $d = 1$.
\end{remark}

The second example shows that this phenomenon is also possible in
derivative cocycles.

\begin{example}\label{maps}
There exists $d$ (\eg $d=9$) and an analytic family of analytic maps
$f_\epsilon: \T^d\to\T^d $ such that:
\begin{itemize}
\item
$f_0$ is an Anosov  linear automorphism of $\torus^d$;
\item
For any $\epsilon$, $Df_\epsilon$ preserves a three dimensional
bundle $E$;
\item
For any $\epsilon$ and any periodic point $p=f^np$,
$Df^n_\epsilon|_{E(p)} $ is conformal in some norm;
\item
For any $\epsilon \ne 0$, $\,Df_\epsilon|_{E}$ is not uniformly
quasi-conformal.

\end{itemize}
\end{example}

\vskip.2cm

\subsection{Construction of Example~\ref{cocycle}}

We pass to a finite power $f\equiv g^N$ of $g$ for which there exist
two fixed points $x_1,x_2$ and two balls $B_1,B_2$ around them such
that for every sequence $\sigma\in \{1,2\}^\natural$ there exists a
point $x^*$ such that
\begin{equation}\label{itinerary}
f^j(x^*) \in B_{\sigma_j}.
\end{equation}
This can be easily arranged using Markov partitions. Of course, the
point $x^*$ is far from being unique, as any point on its  local
stable manifold would have the same itinerary.

We construct the family of
cocycles with required properties over $f$ to illustrate the idea
and then indicate how to carry out similar construction over $g$.
We take
$$
 A_\e(x) = \begin{pmatrix} R_\beta &\e\varphi(x)\\
0&1\end{pmatrix}, \quad\text{where}\quad
R_\beta = \begin{pmatrix} \;\;\;\cos 2\pi\beta&\sin 2\pi \beta\\
-\sin 2\pi \beta &\cos 2\pi\beta\end{pmatrix},
$$
$\beta$ an irrational
number, and $\varphi : \M \to \real^2$ an analytic function satisfying
some properties to be specified later. We observe that
$$
A^n_\e(x) = \begin{pmatrix} R_{n\beta} & \e\widetilde\varphi_n (x)\\
0&1\end{pmatrix} , \quad \text{where}
$$
$$
\widetilde\varphi_n (x)
 = \sum_{j=1}^n R_{(n-j)\beta} \,\varphi (f^j(x))\\
 = R_{n\beta} \sum_{j=1}^n R_{-j\beta} \,\varphi (f^j(x)).
$$
Clearly, the eigenvalues of $A^n_\e(x)$ are $e^{2\pi i n\beta}$,
$e^{-2\pi i n\beta}$, $1$. Since $\beta$ is irrational, all of them
are different, and $A^n_\e(x)$ is diagonalizable. Therefore $A_\e^n (p)$
is conformal in some norm whenever $f^n(p)=p$.
We construct a function $\varphi$ and a point $x^*$ such that
\begin{equation*}
\Big\| \sum_{j=1}^n R_{-j\beta} \,\varphi (f^j (x^*))\Big\| \to\infty
\quad \text{as } n\to\infty ,
\end{equation*}
which implies that  $A_\e$ is not uniformly quasiconformal
for every $\e\ne 0$.

\vskip.1cm

Since the $C^\infty$ case is easier we will discuss it first.
We choose an increasing sequence of integers
 $\J\equiv\{j_k\}_{k=1}^\infty$ such that
$$
\{ j_k\beta  \mod 1 \}\to 0 \quad\text{and hence}\quad
\|R_{-j_k\beta} -\text{Id}\| \to 0 \;\mbox{as}\;\; k\to\infty.
$$
We take a sequence $\sigma$ such that
\begin{equation}\label{sequence}
\sigma_\ell = \begin{cases}
1&\text{if } \ell \in \{j_k\}_{k\in\natural}  \\
\noalign{\vskip6pt} 2&\text{if } \ell\notin \{j_k\}_{k\in\natural}

\end{cases}
\end{equation}
and consider $x^*$ that satisfies \eqref{itinerary} for the sequence
\eqref{sequence}. Now, we choose a function $\varphi$ so that
$$
\varphi (x) =
\displaystyle \binom10 \;\text{ if }x\in B_1 \quad\text{and}\quad
\varphi (x) = \displaystyle \binom00  \;\text{ if } x\in B_2.
$$
Then we have
\begin{equation*}
\Big\| \sum_{j=1}^n R_{-j\beta} \varphi (f^j x^*)\Big\| = \Big\|
\bigg( \sum_{j_k\le n} \mathop{R_{-j_k\beta}}\bigg)
\binom10\Big\| \to\infty \quad \text{as } n\to\infty.
\end{equation*}

The analytic case is slightly more complicated since we cannot use
functions with compact support.
We define a sequence $\J$ in a different way
$$
j\in\J \quad \text{if and only if}\quad
\measuredangle\Big(R_{-j\beta}\binom10,\binom10\Big)<\pi/3.
$$
Consider a corresponding point $x^*$ and an analytic
function $\varphi$ satisfying

$$
 \left|\varphi (x) - \binom10 \right| \le 10^{-8}\; \text{ if } x\in B_1
 \quad \text{and} \quad
 \left|\varphi (x)-\binom00\right| \le 10^{-8} \; \text{ if } x\in B_2.
$$
Such functions can easily be obtained by modifying the corresponding
$C^\infty$ examples.
Since our new sequence has asymptotic density $1/3$ in
$\mathbb N$ a straightforward estimate implies divergence
of corresponding sum.

The point of Example~\ref{cocycle} is that the matrices that diagonalize
$A_\e^n(p)$ are not bounded uniformly in $p$, so while
$A_\e^n(p)$ is conformal in some norm, $K(A_\e^n(p))$ are not
uniformly bounded with respect to the standard norm.

Similar constructions can be carried out for the initial diffeomorphism $g$. Instead of balls $B_1$ and $B_2$ one  needs to work with neigborhoods of $\{x_1, g(x_1),..., g^{N-1}(x_1)\}$ and  $\{x_2, g(x_2),..., g^{N-1}(x_2)\}$ and redefine $\varphi$ accordingly.


\subsection{Construction of Example~\ref{maps}}
Let $B$ and $C$ be hyperbolic integer matrices with determinant 1
such that $r>1$ is an eigenvalue of $C$, and $r e^{\pm2\pi i\beta}$
are eigenvalues of  $B$ for  some irrational $\beta$.
Then
$$
  f_0(x,y)  = (Bx, Cy),
$$
is an Anosov toral automorphism. As in Example~A.1 we pass to a
finite power if necessary, this only changes the value of $r$. To
embed Example~\ref{cocycle} into a diffeomorphism we consider a
perturbation of the form
\begin{equation*}
f_\e (x,y) = (Bx + \e\psi (y),Cy),
\end{equation*}
where $\psi$ takes values in the two dimensional subspace $U$
corresponding to the eigenvalues $re^{\pm 2\pi i\beta}$.
Let $v$ be an eigenvector of $C$ corresponding to $r$. Then
the three dimensional space $W=U \oplus \R v$ is invariant under
$f_\e$. The restriction of the differential $Df_\e$ to $W$ is of the
form $rA_\e$, where
$$
 A_\e(y) = \begin{pmatrix} R_\beta &\frac{\e}{r}\,\frac{\partial \psi}{\partial v}(y)\\
0&1\end{pmatrix}.
$$
We consider this restriction as a cocycle  over Anosov automorphism
$g(y)=Cy$. Thus Example~A.3 is reduced to Example~A.1 provided we
can solve the cohomological equation
$$
\frac{1}{r}\,\frac{\partial\psi}{\partial v}=\varphi,
$$
where $\varphi$ as in  Example A.1. If the
vector $v$ is Diophantine then by a theorem of Kolmogorov this
cohomological equation has a solution $\psi$ if and only if
$\int\varphi=0$ (see e.g.~\cite{Russ}). The vector $v$ is algebraic
and hence Diophantine as an eigenvector of integral matrix.
Additional condition $\int\varphi=0$ can be accommodated since
we have a lot of freedom outside the balls $B_1$ and $B_2$.

The existence of the matrices $B$ and $C$ satisfying the required properties
can be easily seen if one finds an integer coefficients
monic polynomials with corresponding properties.
We communicated the above question to Professor F.~Voloch who kindly
formulated it and posted it at {\tt http://mathoverflow.org}.
In less than a day we obtained several responses~\cite{V} from Buzzard and other participants.
For example, one can take polynomials
$$
   x^3+3x^2+2x-1 \quad \text{and}\quad x^6-2x^4-3x^2-1,
$$
as characteristic polynomials of $B$ and $C$, which gives a 9 dimensional example.


\end{document}